\documentclass[12pt]{amsart}
\usepackage{amsmath}
\usepackage{amsthm}
\usepackage{amssymb}
\usepackage{amscd}
\usepackage{hyperref}
\usepackage{color}

\newcommand{\gothic}{\mathfrak}
\newcommand{\p}{{\gothic{p}}}
\newcommand{\q}{{\gothic{q}}}

\newcommand{\m}{{\gothic{m}}}

\newcommand{\8}{{\infty}}

\newcommand{\ini}{\operatorname{in}}

\newcommand{\Ann}{\operatorname{Ann}}

\newcommand{\depth}{\operatorname{depth}}

\newcommand{\pd}{\operatorname{pd}}
\newcommand{\Spec}{\operatorname{Spec}}
\newcommand{\Supp}{\operatorname{Supp}}

\newcommand{\Tor}{\operatorname{Tor}}

\newcommand{\ass}{\operatorname{Ass}}
\newcommand{\Sing}{\operatorname{Sing}}

\newcommand{\IPD}{\operatorname{IPD}}
\newcommand{\FPD}{\operatorname{FPD}}
\newcommand{\RGD}{\operatorname{SR}}

\newcommand{\Cl}{\operatorname{Cl}}
\newcommand{\Hom}{\operatorname{Hom}}

\newcommand{\length}{\ell}

\renewcommand{\phi}{\varphi}
\DeclareMathOperator{\Pic}{Pic}

\theoremstyle{plain}
\newtheorem{thm}{Theorem}
\newtheorem{cor}[thm]{Corollary}
\newtheorem{prop}[thm]{Proposition}
\newtheorem{lemma}[thm]{Lemma}
\newtheorem{fact}[thm]{Fact}
\newtheorem{conj}[thm]{Conjecture}
\newtheorem{eg}[thm]{Example}

\theoremstyle{definition}
\newtheorem{defn}[thm]{Definition}

\theoremstyle{remark}
\newtheorem{rmk}[thm]{Remark}

\begin{document}
\title[On the (non)rigidity of the Frobenius endomorphism over Gorenstein rings]
{On the (non)rigidity of the Frobenius endomorphism over Gorenstein rings}

\author{Hailong Dao}
\address{Department of Mathematics\\
University of Kansas\\
 Lawrence, KS 66045-7523 USA}
\email{hdao@math.ku.edu}

\author{Jinjia Li}
\address{Department of Mathematical Sciences \\
       Middle Tennessee State University  \\
        Murfreesboro, TN 37132  USA}
        \email{jinjiali@mtsu.edu}

\author{Claudia Miller}
\address{Mathematics Department \\
       Syracuse University  \\
        Syracuse, NY 13244-1150  USA}
\email{clamille@syr.edu}
\date{\today}
\thanks{The third author gratefully acknowledges partial financial support
     from NSA grant \# H98230-06-1-0035. The first author is partially supported by NSF grant 0834050}
\keywords{rigidity, Frobenius endomorphism, $\Tor$, Picard
group, isolated singularity}

\subjclass{Primary: 13A35; Secondary:13D07, 13H10.}
\bibliographystyle{amsplain}

\numberwithin{thm}{section}
\numberwithin{equation}{section}
\begin{abstract} It is well-known that for a large class of  local rings of positive characteristic, including complete intersection rings, the Frobenius endomorphism can be used as a test  for finite projective dimension. In this paper, we exploit this property to study the structure of such rings. One of our results states that the Picard group of the punctured spectrum of such a ring $R$ cannot have $p$-torsion. When $R$ is a local  complete intersection, this recovers (with a purely local algebra proof)  an analogous statement for complete intersections in projective spaces first given in SGA and also a special case of a conjecture by Gabber.  Our method also leads to many simply constructed examples where rigidity for the Frobenius endomorphism does not hold, even when the rings are Gorenstein with isolated singularity. This is in stark contrast to the situation for complete intersection rings.  Also, a related length criterion for modules of finite length and finite projective dimension is discussed towards the end.
\end{abstract}

\maketitle

%%%%%%%%%%%%%%%%%%%%%%%%%%%%%%%%%%%%%%%
\section{Introduction}\label{intro}
%%%%%%%%%%%%%%%%%%%%%%%%%%%%%%%%%%%%%%%

The Frobenius endomorphism for rings of positive characteristic has been
one of the central objects of study in homological commutative
algebra over the past decades. Not only is it a useful tool in
proofs of homological conjectures, but also its intrinsic homological properties
have been shown to have strong connections with the structure of the
ring or of modules over it. In this article we provide several surprising connections, for example the relationship between the ability of Frobenius to detect finite projective dimension of modules and  the torsion part of the divisor class group.

First let us review some  history and notations. In 1969, Kunz characterized regular local rings as those for which the Frobenius endomorphism $f\colon R \to R$ (or equivalently some iteration of it) is flat  \cite[Thm.\ 2.1]{K}.
Since then, a list of papers has yielded further similar homological results for $f$, each analogous to a classical homological result concerning the residue field $k$ (viewed as an $R$-module via $\pi\colon R \to k$); see the survey article \cite{M2} for further details, also \cite{AIM, IS}. We will use the notation ${}^{f^n}\!\! R$ for $R$ viewed as an $R$-module via the $n$th iteration $f^n$ of $f$.

For their celebrated proof of the Intersection Theorem, Peskine and Szpiro generalized one direction of Kunz's result, and shortly thereafter Herzog proved the converse, yielding the following equivalence \cite[Cor.\ 2]{PS1},  \cite[Thm.\ 1.7]{PS2}, \cite[Thm.\ 3.1]{He}:
\vspace{1mm}
\begin{center}
$M$ has finite projective dimension $\iff$ $\Tor^R_i(M,{}^{f^n}\!\! R)=0$ for all $i>0$ and all $n>0$
\end{center}
\vspace{1mm}
This leads one to ask to what extent the module ${}^{f^n}\!\! R$ could function as a test module for finite projective dimension: Is the vanishing of $\Tor^R_i(M,{}^{f^n}\!\! R)$ for just {\it one} value each of $i>0$ and $n>0$ sufficient? In particular, this would imply that the $R$-module ${}^{f^n}\!\! R$ is {\it rigid}, that is, that
\vspace{1mm}
\begin{center}
$\Tor^R_i(M,{}^{f^n}\!\! R)=0 \implies \Tor^R_{i+1}(M,{}^{f^n}\!\! R)=0$
\end{center}
\vspace{1mm}

Several steps toward these goals have been made in recent years.
In the general setting, Koh and Lee proved a {\it finiteness}
result: there is a constant $c(R)$, depending only on the ring
$R$, such that vanishing of $\Tor^R_i(M,{}^{f^n}\!\! R)$ for any
$\depth R+1$ consecutive values of $i>0$ and any one value of
$n\geq c(R)$ implies that $M$ has finite projective dimension
\cite[Prop.\ 2.6]{KL}. In fact, they showed that $\depth R$
consecutive values of $i$ suffice if $R$ is Cohen-Macaulay of
positive dimension. The best possible result (\cite{AM, D3}), however, occurs in the setting
of complete intersection rings: 
\begin{thm}
\label{AMD}(Avramov-Miller, Dutta)
Let $R$ be a local complete intersection and $M$ a finitely generated $R$-module. 
Then the vanishing of $\Tor^R_i(M,{}^{f^n}\!\! R)$ for {\it one}  value each of $i>0$
and $n>0$ implies that $M$ has finite projective dimension.
\end{thm}

Phenomena like this can occur over
non-complete intersection rings as well. In such a case, we call the
corresponding ${}^{f^n}\!\! R$ strongly rigid (which is equivalent
to being rigid when $n \geq c(R)$ by Koh and Lee's result above). See
Definition~\ref{stronglyrigid} and Example~\ref{eg1} ) for known examples. 

\vspace{2.5mm}

In Section~\ref{ci}, we study the properties of Gorenstein local
rings whose corresponding ${}^{f^n}\!\! R$ is strongly rigid. We
show that if $R$ is Gorenstein such that ${}^{f^n}\!\! R$ is
locally strongly rigid (i.e., strongly rigid at the localization
at every prime ideal), then the minimal infinite projective
dimension locus of a module $M$ (see Definition~\ref{IPD}) must be
contained in the set of associated primes of $F^n(M)$ (see
Theorem~\ref{mainthm}). One consequence of this result is the
following characterization for modules of finite projective
dimension:

\vspace{2.5mm}

\noindent {\bf{Corollary~\ref{perfect}}} {\emph{ Let $R$ be
Gorenstein local ring  such that ${}^{f^n}\!\!R$ is locally
strongly rigid for some $n>0$ and $M$ an $R$-module. Then $M$ has
finite projective dimension if and only if $\ass F^n(M)$ is
contained in the finite projective dimension locus of $M$. }}

\vspace{2.5mm}

Note that the class of rings such that ${}^{f^n}\!\!R$ is locally
strongly rigid for all $n>0$ includes, but is strictly bigger than, the class of all local complete intersections, see Example \ref{eg1}.

We also apply Theorem~\ref{mainthm} to prove that the divisor
class groups of certain Gorenstein domains have no $p$-torsion.

\vspace{2.5mm}

\noindent {\bf{Theorem~\ref{class}}} {\emph{ Let $R$ be a Gorenstein local ring such that ${}^{f^n}\!\!R$ is locally
strongly rigid for some $n>0$.  Let $I$ be an reflexive ideal
 such that $I$ is locally free in codimension $2$. Furthermore, assume that $\Hom_R(I,I) \cong R$. 
Let $q=p^n$. Then if $I^{(q)}$
 satisfies Serre's condition $(S_3)$, $I$ must be principal. In particular, the Picard group of the
punctured spectrum of $R$ has no $p$-torsion. In particular, if  $R$ satisfies
condition $(R_2)$, $\Cl(R)$ has no $p$-torsion.}}

\vspace{2.5mm}

In particular, the theorem above shows that the Picard groups of the punctured spectrum of  complete intersection rings cannot have $p$-torsion. For complete intersections in projective spaces, such a  result was first proved by Deligne in SGA 7 (Theorem 1.8, \cite{De}) using sophisticated geometric machinery. We also note that this particular case confirms the positive characteristic case of the following conjecture we learned from \cite{Gab}: 
\begin{conj}(Gabber)\label{GabberConj}
Let $(R,\m)$ be a local complete intersection ring of dimension $3$. Let $U_R=\Spec(R) -\{\m\}$ be the punctured spectrum of $R$. Then $\Pic(U_R)$ is torsion-free. 
\end{conj}

It was implied in \cite{Gab}  that the positive characteristic case is known, but we cannot find a precise reference. In any case, it is worth noting that our proof is purely homological and quite simple.

\vspace{2.5mm}

In Section~\ref{example}, we push the ideas in the previous
section further to construct many examples of Gorenstein
local rings $R$ such that ${}^{f^n}\!\! R$ is not strongly rigid.
In other words, the vanishing of $\Tor^R_i(M,{}^{f^n}\!\! R)$ for
just one value each of $i>0$ and $n>0$ is {\it not} sufficient to
conclude that $M$ has finite projective dimension. Two 
different approaches are used in these constructions. The first
approach boils down to finding an isolated Gorenstein singularity
with torsion class group and applying Theorem~\ref{class} above; see
Example~\ref{veronese}. The second approach takes a 
different route via Lemma~\ref{firstlemma}. 
Obtaining an actual example requires some explicit computations on
the determinantal ring of 2 by 2 minors in 9 variables and hence
is less general than the first approach; see
Example~\ref{counterexamplethm}. The bonus is, however, that these 
have a torsion-free class group. 

In Section~\ref{observe}, we study the connection between (strong)
rigidity and numerical rigidity (see Definitions~\ref{refinerigid} and 
\ref{numericalrigid}) of the Frobenius
endomorphism. The main result we prove there is

\vspace{2.5mm}

\noindent {\bf{Theorem~\ref{connection}}} {\emph{ Let $R$ be a
Cohen-Macaulay local ring with isolated singularity and of
positive dimension. Fix some $n>0$.  If for every nonzerodivisor
$y \in R$, ${}^{f^n}\!\!(R/yR)$ is numerically rigid, then
${}^{f^n}\!\! R$ is strongly rigid against modules of dimension up
to one.}}

\vspace{2.5mm}

The rest of the introduction contains a review of the notation and 
definitions used throughout the paper.
We assume throughout that $R$ is a commutative Noetherian local
ring of prime characteristic $p>0$ and that all $R$-modules $M$ and $N$ are
finitely generated. The Frobenius endomorphism $f\colon R \to R$ is defined by $f(r)=r^p$ for $r \in
R$; its self-compositions are given by $f^n(r) = r^{p^n}$.
Restriction of scalars along each iteration $f^n$ endows $R$ with a new $R$-module structure, denoted by ${}^{f^n}\!\! R$.

The Frobenius functor, introduced by Peskine and Szpiro in \cite{PS2}, is given by base change along
the Frobenius endomorphism: $F_R(M) = M\otimes_R {}^{f}\! R$ for any $R$-module $M$.
Its compositions are given by $F_R^n(M) = M\otimes_R {}^{f^n}\!\! R$,
namely base change along the compositions $f^n$ of $f$. We omit the subscript $R$ if there is no ambiguity about $R$.
Note particularly that the module structure on $F^n(M)$ is via usual multiplication in $R$ on the right hand factor of the tensor product. The values of the derived functors $\Tor^R_i (M,{}^{f^n}\!\! R)$ are similarly viewed as $R$-modules via the target of the base change map $f^n$.

It is easy to verify that $F^n(R) \cong R$ and that for cyclic modules
$F^n(R/I) \cong R/I^{[q]}$, where $q=p^n$ and $I^{[q]}$ denotes the
ideal generated by the $q$-th powers of the generators of $I$.
For convenience, we frequently use $q$ to denote the power $p^n$,
which may vary.

In the sequel, we use $\length(M)$ and $\pd M$ to denote the
\emph{length} and \emph{projective dimension}, respectively,  of the module
$M$. By codimension of $M$,  we mean $\dim R-\dim M$. We use the notation $\bold x$ for a sequence of elements of $R$ and often write simply $R/\bold x$ for $R/(\bold x)$ to save space. Likewise, $\bold x^q$ denotes the ideal generated by the $q$th powers of the sequence $\bold x$, {\it not} the $q$th power of the ideal $\bold x$.

%%%%%%%%%%%%%%%%%%%%%%%%%%%%%%%%%%%%%%%%%%%%%%%%%%%%%%%%%%%%%%%%%%%%%%%
\section{Strong rigidity of Frobenius  and torsion elements in divisor class groups} \label{ci}
%%%%%%%%%%%%%%%%%%%%%%%%%%%%%%%%%%%%%%%%%%%%%%%%%%%%%%%%%%%%%%%%%%%%%%%

In this section we investigate the consequences of the phenomenon that over certain rings the Frobenius map can be used to test for finite projective dimension (e.g., over complete intersection rings). This work enables us to prove strong results about torsion elements in the class groups of complete intersection rings and also allows us to construct counterexamples to such phenomena over non-complete intersection rings. We begin with some convenient definitions to facilitate the discussion.

\begin{defn}\label{stronglyrigid} An $R$-module $N$ is called {\it strongly rigid} if  for any integer $i$ and any finitely generated $R$-module $M$, $\Tor_i^R(M,N)=0$ implies $\pd_RM <\infty$. The module $N$ is called {\it locally strongly rigid} if $N_{\p}$ is strongly rigid for all $\p \in \Spec R$.
\end{defn}

\begin{eg}\label{eg1}
If $R$ is a local complete intersection ring, then ${}^{f^n}\!\!R$ is
locally strongly rigid for all $n$, see Theorem \ref{AMD}. For any local Cohen-Macaulay ring $R$ of dimension at most $1$, there is a number
$c(R)$ such that for any $n \geq c(R)$, ${}^{f^n}\!\!R$ is
strongly rigid  by virtue of Koh
and Lee's result mentioned in the introduction. In particular, when $(R,\m)$ is Artinian and $\m^{[p]}=0$, then  ${}^{f^n}\!\!R$
is (locally) strongly rigid for all $n$, cf. \cite[2.2.8]{M2}. 
\end{eg}

\noindent We also cite the following definition (see, for example, \cite{Dao}):
\begin{defn}\label{IPD} Let $M$ be an $R$-module. One defines the infinite projective dimension locus of $M$ as
\[\IPD(M)=\{\p \in \Spec R | \pd_{R_{\p}} M_{\p}=\8\}\]
Similarly, define $\FPD(M)$ to be the finite projective dimension locus of $M$.
Finally, we define the $n$-strong rigidity locus of $R$ as
\[\RGD_n(R)=\{\p \in \Spec R | ,{}^{f^n}\!\!R_{\p}\text{ is strongly rigid}\}\]
\end{defn}

The following standard facts, which we state without proof, will be used frequently:

\begin{fact}
\label{fact}
Let $f\colon R \to S$ be a ring homomorphism and $\p$ a prime ideal of $S$.
Then for each $i\geq 0$ and $R$-module $M$ there is a natural isomorphism
\[
\Tor_i^R(M,S)_\p
\cong
\Tor_i^{R_{f^{{\textrm{--}}1}\!(\p)}}(M_{f^{{\textrm{--}}1}\!(\p)}, S_\p)
\]
Furthermore, if $f$ is the Frobenius endomorphism of $R$, then
$f^{{\textrm{--}}1}\!(\p) = \p$ and
$R_{f^{{\textrm{--}}1}\!(\p)} \to S_\p$ is the Frobenius endomorphism of $R_\p$.
\end{fact}

\begin{thm}
\label{mainthm}
Let $R$ be a Gorenstein local ring and $M$ an $R$-module. Then
\[\min \IPD(M) \cap \RGD_n(R) \subseteq \ass F^n(M)\]
In particular, if ${}^{f^n}\!\!R$ is locally strongly rigid, then
\[\min \IPD(M)  \subseteq \ass F^n(M)\]
\end{thm}

\begin{proof}
Since $R$ is Gorenstein, by the Cohen-Macaulay approximation due
to Auslander and Buchweitz \cite[1.8]{AB}, there is a short exact
sequence
\[
0 \to M \to Q \to N \to 0,
\]
where $\pd Q< \8$ and $N$ is maximal Cohen-Macaulay. Tensoring
with the Frobenius endomorphism, we have an embedding

\begin{equation}{\label{basic}}
0 \to \Tor^R_1 (N,{}^{f^n}\!\! R) \to F^n(M).
\end{equation}

Take any $\p \in \min \IPD(M) \cap \RGD_n(R)$; then $\pd_{R_{\p}}
M_{\p}=\8$ and ${}^{f^n}\!\!R_{\p}$ is strongly rigid. It follows
that $\pd_{R_{\p}} N_{\p}=\8$ and therefore that 
$\Tor^{R_{\p}}_1 (N_{\p},{}^{f^n}\!\! R_{\p})\neq
0$. On the other hand, since $\p$ is minimal in the infinite
projective dimension locus of $M$, $\pd_{R_{\q}}M_{\q} < \8$ for
any prime $\q \subsetneq \p$, whence $\pd_{R_{\q}}N_{\q} < \8$ and
so $\Tor^{R_{\q}}_1 (N_{\q},{}^{f^n}\!\! R_{\q})=0$. Therefore, the
length of $\Tor^{R_{\p}}_1 (N_{\p},{}^{f^n}\!\! R_{\p})$ must be
finite.

Localizing (\ref{basic}) at $\p$, we have an exact sequence
\[0 \to \Tor^{R_{\p}}_1 (N_{\p},{}^{f^n}\!\! R_{\p}) \to F^n(M)_{\p}.
\]
This implies that $\depth F^n(M)_{\p}=0$. Hence $\p \in  \ass F^n(M)$.
\end{proof}

\begin{cor}
\label{perfect} Let $R$ be  a Gorenstein local ring  such that
${}^{f^n}\!\!R$ is locally strongly rigid for some $n>0$ and $M$
an $R$-module. The following are equivalent:
\begin{enumerate}
   \item $M$ has finite
projective dimension;
       \item  $\ass F^n(M) \subseteq \FPD(M)$.
\end{enumerate}
\end{cor}

As an immediate consequence, we obtain the following special case with
simpler hypotheses. Here, $\Sing (R)$ denotes the singular locus of $R$. Note particularly that the hypothesis
$(\min \Supp M) \cap \Sing(R)=\emptyset$ holds, for example, when $\dim M > \dim \Sing(R)$.

\begin{cor}
\label{perfect1}
Let $R$ be  a Gorenstein local ring  such that
${}^{f^n}\!\!R$ is locally strongly rigid for some $n>0$ (e.g., if $R$ is a local complete intersection) 
and $M$ an $R$-module such that  $(\min \Supp M) \cap \Sing(R)=\emptyset$. If $F^n(M)$ has no embedded primes, then $M$ has finite projective dimension. In particular, if $F^n(M)$ is Cohen-Macaulay, then $M$ is perfect.
\end{cor}

\begin{proof}
It suffices to note that
\[
\ass F^n(M)=\min \Supp F^n(M)=\min \Supp M\subseteq \Spec R \backslash \Sing(R)
\subseteq \FPD(M)
\]
where the first equality is by the assumption that $F^n(M)$ has no embedded primes, the second is well-known (see \cite{PS2}, for example) and the first containment follows from the hypothesis.
\end{proof}

\begin{rmk}
If $R$ is reduced, we do not know if the condition ``$(\min \Supp
M) \cap \Sing(R)=\emptyset$'' in Corollary~\ref{perfect1} can be replaced by
the simpler condition ``$\dim M>0$''. However, this is impossible
when $R$ is not reduced (see \cite[2.1.7]{M2}, for example).
\end{rmk}

We now give an application of Theorem~\ref{mainthm} to divisor
class groups. In the sequel, we use $\Cl(R)$ to denote the divisor
class group of $R$. We refer to \cite{F} for the definition and
basic facts about $\Cl(R)$ and the Picard groups and to \cite{BH} for Serre's conditions
($R_n$) and ($S_n$).

\begin{thm}\label{class}
Let $R$ be a
Gorenstein local ring such that ${}^{f^n}\!\!R$ is locally
strongly rigid for some $n>0$.  Let $I$ be an reflexive ideal
 such that $I$ is locally free in codimension $2$. Furthermore, assume that $\Hom_R(I,I) \cong R$. 
Let $q=p^n$. Then if $I^{(q)}$
 satisfies Serre's condition $(S_3)$, $I$ must be principal. In particular, the Picard group of the
punctured spectrum of $R$ has no $p$-torsion. If, furthermore, $R$ satisfies
condition $(R_2)$, then $\Cl(R)$ has no $p$-torsion.
\end{thm}

\begin{proof}
We may assume $\dim R\geq 3$. Assume that $I$ is not principal, then it follows that $\pd
I=\8$ (see \cite[Corollary 11]{Bra}\cite[Chapter VII, \textsection 4, no. 7, Corollary
2]{Bo}). We claim that one can always write $I=(a):(b)$ for $a,b
\in R$. Here is a quick proof: choose $a$ such that $a$ generates 
$I$ at the minimal primes of $I$. Pick an irredundant primary
decomposition of $(a)$; it can be written as $I\cap J$ (if $I=(a)$
we are done). Choosing $b$ in $J$ but not in any minimal prime
of $I$, one can show that $I = (a):(b)$. By the short exact
sequence
\[
0 \to R/(a:b) \stackrel{b}{\to} R/(a) \to R/(a,b) \to 0
\]
we have $\IPD(I)=\IPD(R/(a,b))$.
Thus, for any $\p \in \min \IPD(I)$,
\[
\p \in \min \IPD(R/(a,b))
\]
and so by Theorem~\ref{mainthm},
\[
\p \in \ass (F^n(R/(a,b)))= \ass (R/(a^q, b^q))
\]
Localize the short exact sequence
\[
0 \to R/(a^q:b^q) \to R/(a^q) \to R/(a^q,b^q) \to 0
\]
at $\p$, and observe that  $(a^q:b^q)=I^{(q)}$. From the fact that
\[
\depth (R/(a^q,b^q))_{\p}=0
\]
we get that $\depth I^{(q)}_{\p} \leq 2$. On the other hand, since $I$ is locally free in codimension $2$, $\dim R_{\p}\geq 3$. So, $I^{(q)}$ does not satisfy $(S_3)$, and our first assertion is proved. The last two statements follow immediately (note that if $R$ is $(R_2)$ then $R$ is automatically normal). 
\end{proof}

As a corollary we can recover a notable result about torsion elements in the Picard groups of  complete intersections.

\begin{cor}
\label{classgroup}
Let $R$ be an equicharacteristic local complete intersection ring of dimension at least $3$. Then the Picard group of the punctured spectrum of $R$ is torsion-free.  If, furthermore, $R$ satisfies condition $(R_2)$, then the class group of $R$ is torsion-free.

Let $X$ be a complete intersection variety of dimension at least $2$ in the projective space over a field. The Picard group of $X$ modulo the hyperplane section is torsion-free.
\end{cor}

\begin{proof}
Let $p$ be the characteristic exponent of $R$ (so it is $1$ if the characteristic of $R$ is $0$).
The fact that the Picard group or $\Cl(R)$ has no element whose order is relatively prime to $p$ was well-known, see  \cite{Ro}. Theorem~\ref{class} takes care of the $p$-torsion elements. The second half of the corollary follows by applying the first to the local ring at the origin of the affine cone over $X$.
\end{proof}

\begin{rmk}
The second half of the corollary was first proved by Deligne
\cite{De}, and another proof was given by B\v{a}descu \cite[Theorem
B]{Ba}. As far as we know, ours is the first algebraic proof.
\end{rmk}

\begin{eg}
The conditions $\dim R \geq 3$  and ($R_2$) in the corollary cannot 
be weakened. Let $R = k[[x,y,z]]/(xy-z^2)$ where $k$ is an algebraically closed field of characteristic other than $2$. Then $\dim R=2$ and
$R$ is regular in codimension $1$, but $\Cl(R) \cong
\mathbb{Z}/(2)$ (see, for example, Proposition 11.4 of \cite{F}).
\end{eg}

%%%%%%%%%%%%%%%%%%%%%%%%%%%%%%%%%%%%%%%%%%%%%%%%%%
\section{Examples of nonrigidity}\label{example}
%%%%%%%%%%%%%%%%%%%%%%%%%%%%%%%%%%%%%%%%%%%%%%%%%%

In this section we construct plenty of examples of a Gorenstein
ring $R$ in positive characteristic such that ${}^{f^n}\!\! R$ is
not (strongly) rigid. This is in stark contrast to the situation
for complete intersection rings, where the strong rigidity of
${}^{f^n}\!\! R$ is known to hold. Our constructions take two
completely different approaches. The first approach (see
Example~\ref{veronese}) provides the desired examples with torsion
divisor class groups. This can be viewed as a natural consequence
of Theorem~\ref{class}. The second approach (see
Example~\ref{counterexamplethm}), on the contrary, provides the
desired examples with torsion-free divisor class groups.

First we isolate a consequence of  Theorem~\ref{class}:

\begin{cor} 
\label{torsionex} 
Let $R$ be a local, Gorenstein domain with isolated singularity. Suppose 
that $\dim R \geq 3$ and $\Cl(R)$ has a torsion element of order $l$ that 
satisfies ($S_3$). Then  ${}^{f^n}\!\!R$ is not strongly rigid for any $n$ such 
that $p^n \equiv 1$ or $0$ modulo $l$. In particular, if $l=2$, then ${}^{f^n}\!\!R$ 
is not strongly rigid for any $n$ and not rigid for $n\gg 0$.
\end{cor}

\begin{proof} 
Let $I$ be a reflexive ideal which represents an $l$-torsion 
element in $\Cl(R)$ and $q=p^n$.  Then the ideal $J = I^{(q)}$ is isomorphic 
to $I$ or $R$, both of which
satisfy $(S_3)$, contradicting Theorem~\ref{class}. When $l=2$,
for any $n$, $q = p^n$ is congruent to $0$ or $1$ modulo $2$. The
last statement follows from Example~\ref{eg1}.
\end{proof}

\begin{eg}
\label{veronese}
It is not hard to find examples of isolated Gorenstein singularities with torsion 
class group. Let $S= k[x_1,\cdots,x_d]$ and $l$ be an integer. Let $T$ be 
the $l$-Veronese subring of $S$ and $R$ be the local ring at the homogeneous 
maximal ideal of $T$. Then one can show that $\Cl(R) =\Cl(T) = \mathbb{Z}/(l)$ 
using Theorem 1.6 of \cite{Wa}. The ring $R$  obviously  has isolated singularity, 
as it is the local ring at the origin of the cone over a smooth projective variety. 
Also, $R$ will be Gorenstein as long as $l \mid d$. Finally, let $I$ represent 
the generator of $\Cl(T)$. It is easy to see that the cyclic cover of $T$ corresponding 
to $I$ is $S$, so $I$, and therefore the generator of $\Cl(R)$, is Cohen-Macaulay. 
In particular, it will be ($S_3$). So all of the conditions of Corollary~\ref{torsionex} can be 
satisfied easily.
\end{eg}

For the rest of this section we will take another approach to
construct examples of nonrigidity in which the rings have 
{\it torsion-free} divisor class groups. The
following result gives a general technique for finding such
examples:

\begin{lemma}
\label{firstlemma}
Let $(R,\m)$ be a Gorenstein ring with isolated
singularity and positive dimension. The following are equivalent:
\begin{enumerate}
   \item ${}^{f^n}\!\!R$ is strongly rigid.
   \item  For any $R$-module $L$ with infinite
projective dimension,  $\depth F^n(L)=0$
\end{enumerate}

\end{lemma}

\begin{proof}[Proof of Lemma~\ref{firstlemma}]
Assume (1). Then (2) is a consequence of Corollary~\ref{perfect}.
Now assume (2). Let $L$ be a module of infinite projective
dimension. It is enough to prove that $\Tor_1^R(L,{}^{f^n}\!\!R)
\neq 0$. Consider the exact sequence
\[
0 \to L_1 \to Q \to L \to 0
\]
where $Q$ is free and $L_1$ is the first syzygy of $L$. If $\Tor_1^R(L,{}^{f^n}\!\!R) = 0$, then by tensoring with ${}^{f^n}\!\!R$
one get
\[
0 \to F^n(L_1) \to Q \to F^n(L) \to 0.
\]
But since $\pd_R L= \pd_R L_1 = \infty$, one has $\depth F^n(L_1) = \depth F^n(L) =0$. 
Since $\depth Q =\dim Q >0$, this is a contradiction.
\end{proof}

We also need the following crucial observation.

\begin{lemma}
\label{secondlemma} 
Let $k$ be a field of characteristic $p>0$. Let
$A$ denote the determinantal ring $k[X]/I_2$ where $X=(X_{ij})$
is a 3 $\times$ 3 matrix of indeterminates and $I_2$ is the ideal of
$k[X]$ generated by all the 2 $\times$ 2 minors of $X$. Let $x_{ij}$
denote the images of $X_{ij}$ in $A$. Let
$L=A/(x_{11},x_{12})$. Then $\depth F^n(L)>0$ for all $n>0$ and $\pd L = \8$.
\end{lemma}

\begin{proof}
Let $\delta_{ij}$ denote the minors of $X$ corresponding to
$X_{ij}$ and $I$ be the ideal of $k[X]$ generated by
$X_{11}^n,X_{12}^n$ and all the $\delta_{ij}$'s. We prove that for
any field $k$ (we do not need to assume that $k$ has prime
characteristic!) and any $n\geq 2$, $x_{33}$ is a nonzerodivisor for $A/(x_{11}^n,x_{12}^n) \cong k[X]/I$. In the
following paragraph, we refer the reader to \cite[15.2-4]{E} for notations
and terminologies regarding Gr\"{o}bner basis.

We fix a \emph{reverse lexicographic order} $>$ on the monomials
such that
\[X_{11}>X_{12}>X_{13}>X_{21}>X_{22}>X_{23}>X_{31}>X_{32}>X_{33}\]
Using \emph{Buchberger's Algorithm}, one can produce a Gr\"{o}bner
basis for $I$ consisting of all the $\delta_{ij}$'s, $X_{11}^n$,
$X_{12}^n$ and all the monomials of the form
$X_{11}^lX_{12}^{n-l}X_{22}^sX_{32}^t$ where $l$ runs through
$1,2,...,n-1$ and $s,t$ run through all positive integers such
that $s+t=l$. Therefore the \emph{initial ideal} of $I$ (henceforth $\ini(I)$)
does not contain any monomial divisible by $X_{33}$. Assume for some
$g \in k[X]$, $X_{33}g\in I$. Let $g_0$ be the \emph{remainder} of
$g$ (with respect to the generators of $I$) in a \emph{standard
expression} obtained by performing the \emph{division algorithm}.
If $g_0 \neq 0$, then $X_{33}g_0 \neq 0$ since $k[X]$ is a domain.
On the other hand, since $X_{33}g_0\in I$, at least one of the
monomials of $X_{33}g_0$ is in $\ini(I)$. Thus, at least one of
the monomials of $g_0$ is in $\ini(I)$. This contradicts the fact
that $g_0$ is a nonzero \emph{remainder}. Thus $g_0=0$ and $g \in
I$. It follows that $x_{33}$ is a nonzerodivisor for $k[X]/I$.

Finally, we show $\pd L = \8$. Assume that $\pd L < \8$, i.e., the
ideal $(x_{11},x_{12})$ is of finite projective dimension. By a
result of MacRae \cite[Cor.\ 4.4]{MA}, two-generated ideals of
finite projective dimension have the form $a(b,c)$, where $a$ is a
nonzerodivisor and $b,c$ form a regular sequence. But if
$(x_{11},x_{12})=a(b,c)$ for such $a$, $b$ and $c$, since the degree of $x_{11}$ is one, $a$
is forced to be a unit (otherwise, $(x_{11}, x_{12})$ would be a principal ideal which is impossible).
Therefore $(x_{11},x_{12})=(b,c)$. But since
$x_{11}x_{22}-x_{21}x_{12}=0$, $(x_{11},x_{12})$ cannot be an
ideal generated by a regular sequence of two elements. This is a contradiction.
\end{proof}

Combining the two lemmas above, we get the following example. Note
that the divisor class group of the ring in this example is
isomorphic to $\mathbb Z$ \cite[7.3.5]{BH}, which is torsion-free.

\begin{eg}
\label{counterexamplethm}
Let $R$ be the localization of the determinantal ring $A$ as in Lemma~\ref{secondlemma} with respect to the maximal ideal $(X)$.
Then ${}^{f^n}\!\!R$ is not strongly rigid for any $n$.
\end{eg}

\begin{rmk} In view of Koh-Lee's theorem mentioned in the Introduction, 
Example~\ref{counterexamplethm} immediately yields the nonrigidity of 
${}^{f^n}\!\! R$ for any $n\geq c(R)$ (cf.\ Example~\ref{eg1}). 
But in fact, with a little further computation, the reader can check 
that this example yields nonrigidity for {\it all} $n>0$:
indeed, the module $N$ of infinite projective dimension constructed 
in Theorem~\ref{mainthm} by taking for the 
module $M$ the module $L$ of Lemma~\ref{secondlemma} satisfies  
$\Tor^R_1 (N,{}^{f^n}\!\! R)=0$ by the argument in the proof, but it can be shown that 
in fact $\Tor^R_2 (N,{}^{f^n}\!\! R)\neq 0$. 

We point out that we do not know of any example showing that ${}^{f^n}\!\! R$ is 
not (strongly) rigid when $\dim R=0$ or against a module $M$ of  finite length. 
See, however, the discussion at the end of Section~\ref{observe}. 
\end{rmk}

%%%%%%%%%%%%%%%%%%%%%%%%%%%%%%%%%%%%%%%%%%%%%%%%%%%%
\section{Some further observations} \label{observe}
%%%%%%%%%%%%%%%%%%%%%%%%%%%%%%%%%%%%%%%%%%%%%%%%%%%%
Throughout this section, $d$ will always be the dimension of the
ring and $n$ always denotes some positive integer. We know from
the previous section that $R$ could fail to be strongly rigid when
the ring $R$ is no longer complete intersection. However, we
still hope that to some extent such a property could hold over
non-complete intersection rings. In particular, we do not know any
example showing that ${}^{f^n}\!\! R$ is not rigid when $\dim R=0$
or against a module $M$ of  finite length.

We first make two more definitions, the first of which is just a
refinement of the definition of strongly rigidity of ${}^{f^n}\!\!
R$.

\begin{defn}\label{refinerigid} Let $h$ be a nonnegative integer. ${}^{f^n}\!\! R$ is called \emph{strongly rigid against modules of dimension at most $h$}, if for any integer $i$ and any finitely generated module $M$ of dimension at most $h$,  $\Tor_i^R(M,{}^{f^n}\!\! R)=0$ implies $\pd_RM <\infty$.
\end{defn}

\begin{defn}\label{numericalrigid} ${}^{f^n}\!\! R$ is called \emph{numerically rigid} if for any $R$-module $M$ of finite length, $\length (F^n(M))=p^{nd}\length(M)$ implies
$\pd_RM <\infty$. 
\end{defn}

The second definition above is motivated by the following characterization for modules of finite projective dimension and finite length over complete intersection rings.\\

\vspace{1mm}

\noindent {\bf{Theorem} \cite{D1, M2}} {\emph{Let $R$ be a complete intersection ring in characteristic $p$ and $M$ an $R$-module of finite length. Then the following are equivalent:\\
\noindent(1) $M$ has finite projective dimension,\\
\noindent(2) $\length (F^n(M))=p^{nd}\length(M)$ for all $n>0$,\\
\noindent(3) $\length (F^n(M))=p^{nd}\length(M)$ for some $n>0$.\\}
}

\vspace{1mm}

 The implication (3)$\Rightarrow$(1) in the above theorem simply
says that if $R$ is a complete intersection ring, then
${}^{f^n}\!\! R$ is \emph{numerically rigid} for any $n$. When $R$
is no longer a complete intersection ring, it is an open question
whether ${}^{f^n}\!\! R$ could still be numerically
rigid\footnote{the implication (1)$\Rightarrow$(2) in the theorem
fails even over Gorenstein rings, see \cite{MS}}. In fact, such a
question is closely related to the rigidity question discussed
earlier in this paper. The goal of this section is to explore the
connections between them.

\vspace{2.5mm}

The following technical result plays a crucial role here. Recall
that if $\length(M\otimes N)<\8$ and $\pd N <\8$, then $\chi(M,N)\
{ \stackrel{\textrm{def}}{= }} \ \sum_{j=0}^{\pd
N}(-1)^{j}\length(\Tor_j^R(M,N))$.
\begin{prop}
\label{prop1}
Let $R$ be a Noetherian local Cohen-Macaulay
ring of positive dimension and of characteristic $p>0$. Let $M$ be an $R$-module of codimension $c$. Suppose $\dim M>0$ and $R_{\p}$ is a complete intersection ring for every minimal prime $\p$ of $M$. Then
\[
\length(F^n_{R/\bold x}(M/\bold x M)) \geq q^c \chi(M, R/\bold x), \forall n>0,
\]
for any s.o.p. $\bold x$ of $F^n(M)$ which
is also $R$-regular. Moreover, equality holds for some $n>0$
if and only if: (1) $F^n(M)$ is Cohen-Macaulay and (2) for every
minimal prime $\p$ of $M$, $\pd_{R_\p} M_{\p}<\8$.
\end{prop}

For the proof the properties of the higher
Euler characteristics of Koszul complexes are used in an essential
way. We recall some terms and results here.

For a pair of modules $M$, $N$ such that $\length(M\otimes N)<\8$ and $\pd N <\8$,
the {\it $i$th higher Euler characteristic} is defined
by the following formula
\[
\chi_i(M,N)=
\sum_{j=i}^{\pd N}(-1)^{j-i}\length(\Tor_j^R(M,N)).
\]
By convention, $\chi(M,N)=\chi_0(M,N)$. Some standard facts about
$\chi$ and $\chi_i$ can be found in \cite{L, S}. In this paper, we
particularly need the following two well-known results  due to Serre and
Lichtenbaum  \cite[Lemma 1 and Theorem 1]{L}, \cite{S}.

\begin{lemma}\label{Li1} Let $M$ be an $R$-module and $\bold x=\{x_1,x_2,...,x_c\}$
an $R$-sequence such that $\length(M/\bold xM)<\8$. Then $\chi(M,
R/\bold x)\geq 0$, with the equality holding if and only if $\dim M < c$.
\end{lemma}

\begin{thm} \label{Li2} Let $M$ be an $R$-module and $\bold x$
an $R$-sequence such that $\length(M/\bold xM)<\8$. Then for any
$i>0$, $\chi_i(M, R/\bold x) \geq 0$, with the equality holding
if and only if $\Tor_i(M,R/\bold x)=0$ {\rm{(}}and hence $\Tor_j(M,R/\bold x)=0$ for all
$j\geq i${\rm{)}}.
\end{thm}

\begin{proof}[Proof of Proposition~\ref{prop1}]
Since $\Supp M=\Supp F^n(M)$ (see \cite{PS2}), $\min \Supp
F^n(M)=\min \Supp M$. We have
\begin{align*}
\length(F^n_{R/\bold x}(M/\bold x M))&=\length(F^n(M)\otimes_R
{R/\bold x }) \\
&\geq \chi(F^n(M), R/\bold x)\\
&=\sum_{\p \in \min \Supp M }\ \length(F^n(M)_{\p}) \chi(R/{\p},
R/\bold x)\\
&=\sum_{\p \in \min \Supp M }\ \length(F^n_{R_{\p}}(M_{\p}))
\chi(R/{\p},
R/\bold x)\\
&\geq \sum_{\p \in \min \Supp M }\ q^c\length(M_{\p}) \chi(R/{\p},
R/\bold x )\\
&= q^c \chi(M, R/\bold x)
\end{align*}
where the first inequality holds since $\chi_1(F^n(M), R/\bold x)\geq 0$ by Theorem~\ref{Li2}, the second and fourth equalities by Lemma~\ref{Li1}, and the second inequality is a result over complete intersection  rings \cite[Thm.\ 1.9]{D1} (note that $R_\p$ is complete intersection by the hypotheses).

Therefore, furthermore, equality holds if and only if $\chi_1(F^n(M), R/\bold x)=0$ and
$\length(F^n_{R_{\p}}(M_{\p}))=q^c\length(M_{\p})$ for every
minimal prime $\p$ of $M$. The former is equivalent to $F^n(M)$ being
Cohen-Macaulay by Theorem~\ref{Li2} and the latter is equivalent to $M_{\p}$ having
finite projective dimension over $R_{\p}$ by \cite[Thm.\ 5.2.2]{M2} since $R_{\p}$ is
complete intersection.
\end{proof}

\begin{thm}
\label{connection} Let $R$ be a Cohen-Macaulay local
ring with isolated singularity and of positive dimension. Fix some $n>0$.  If for every nonzerodivisor
$y \in R$, ${}^{f^n}\!\!(R/yR)$ is numerically rigid, then ${}^{f^n}\!\! R$ is strongly rigid against modules of dimension at most one.
\end{thm}

\begin{proof} Let $M$ be an $R$-module of dimension at most one. Assume
${}^{f^n}\!\!(R/yR)$ is numerically rigid for every nonzerodivisor
$y \in R$. We want to prove that for any $i>0$, $\Tor_i(M,{}^{f^n}\!\!
R)=0$ implies $\pd M < \8$. Let $\bold x=\{ x_1,...,x_{d-1} \}$ be an
$R$-sequence contained in $\Ann M$.
We may assume that $i=1$ by replacing $M$ by its $(i-1)$th syzygy
over the ring $R/(x_1,...,x_{d-1})$ and using that
$\Tor_i(R/(x_1,...,x_{d-1}) ,{}^{f^n}\!\! R)=0$ for all $i>0$ as
$\pd_R R/(x_1,...,x_{d-1}) < \infty$.

Letting $K$ be the first syzygy of $M$ as an $R/(x_1,...,x_{d-1})$-module,
we get a short exact sequence
\[
0 \to F^n(K) \to F^n((R/(x_1,...,x_{d-1}))^t) \to F^n(M) \to 0.
\]
It follows that $F^n(K)$ is a Cohen-Macaulay module
of dimension one. Hence by Proposition~\ref{prop1} (note that $R$ has
isolated singularity),
one has
$\length(F^n_{R/yR}(K/yK))=q^{d-1}\chi(K, R/yR)$ for every $y \in R$
which is regular on both
$K$ and $R$. Therefore, $\length(F^n_{R/yR}(K/yK))=q^{d-1}\length(K/yK)$.
Since we assume  ${}^{f^n}\!\!(R/yR)$ is numerically rigid,
$K/yK$ has finite projective dimension over $R/yR$. Thus $K$ has
finite projective dimension over $R$, whence $M$ does too by
the long exact sequence of Tors against the residue field k.
\end{proof}

\begin{rmk}
For the determinantal ring $R=k[X]/I_2$ used in
Section~\ref{example}, it was shown there that ${}^{f^n}\!\!R$ is
not strongly rigid against modules of dimension at most $5$ for any
$n$. In fact, we can also modify the example a little bit to show
that it is not strongly rigid against modules of dimension at most 
$3$. For $k$ of arbitrary characteristic though, we do not know if
${}^{f^n}\!\!R$ is strongly rigid against modules of dimension at most 0, 1, or 2. 
However, in characteristic $2$ we have an example which shows that  
${}^{f^1}\!\!R$ is not strongly rigid against modules
of dimension 1. In fact, if we set $k=\mathbb{Z}/2\mathbb{Z}$ and take the module
$N=R/(x_{12},x_{13},x_{21},x_{23},x_{31},x_{32})$, then it is easy to check 
that $\dim N=1$, $\depth F(N)=1$ and $\pd N=\8$. Taking an $R$-sequence
$x_1,x_2,x_3,x_4$ contained in the annihilator of $N$ and
embedding $N$ into a module of finite projective dimension over
$R/(x_1,x_2,x_3,x_4)$ (via the Auslander-Buchweitz short exact sequence
again), the cokernel of this embedding gives such an example.
Therefore, by Theorem~\ref{connection}, we also obtain an example
of Gorenstein ring $R$ in characteristic 2 for which the corresponding
$R$-module ${}^{f^1}\!\!R$ is not numerically rigid.
\end{rmk}

\specialsection*{ACKNOWLEDGEMENTS}
We thank Craig Huneke for crucial remarks which connected our works and 
Srikanth Iyengar who carefully read an earlier version of this manuscript and suggested a better statement for Corollary~\ref{perfect}.  We also thank the anonymous referee for many helpful comments.

\end{document}